\documentclass{article}
\usepackage{amsfonts}
\usepackage{amsmath}
\usepackage{hyperref}

\newtheorem{theorem}{Theorem}

\newtheorem{corollary}[theorem]{Corollary}

\newtheorem{lemma}[theorem]{Lemma}

\newtheorem{remark}[theorem]{Remark}

\newenvironment{proof}[1][Proof]{\noindent\textbf{#1.} }{\ \rule{0.5em}{0.5em}}

\begin{document}

\title{On the application of a hypergeometric identity to generate
generalized hypergeometric reduction formulas }
\author{J.L. Gonz\'{a}lez-Santander, \\
Departament of Mathematics. University of Oviedo. \\
C/ Federico Garc\'{\i}a Lorca 18, Oviedo 33007, Asturias, Spain. }
\maketitle

\begin{abstract}
We systematically exploit a new generalized hypergeometric identity to
obtain new hypergeometric summation formulas. As a consistency test,
alternative proofs for some special cases are also provided. As a byproduct
new summation formulas involving finite sums involving the psi function, or
a recursive formula for the Bateman's G function are derived. Finally, all
the results have been numerically checked with MATHEMATICA.
\end{abstract}



\section{Introduction}

Mathematical applications of the generalized hypergeometric functions are
abundant in the existing literature. For instance, a variety of problems in
classical mechanics and mathematical physics lead to Picard-Fuchs equations.
These equations are frequently solvable in terms of generalized
hypergeometric functions \cite{Berglund}. Also, many combinatorial
identities, especially ones involving binomial and related coefficients, are
special cases of hypergeometric identities \cite[Sect. 2.7]{Andrews}.
Another example is given in the calculation of the moments of certain
probability distributions, which are given in terms of generalized
hypergeometric functions \cite{ProbabilityJL}. Therefore, the calculation of
generalized hypergeometric functions for particular values of the argument
and the parameters in terms of more simple functions is of great relevance.
The classical compilation of these summation formulas is found in \cite[%
Chap. 7]{Prudnikov3}. A revision, as well as an extension of these tables,
was carried out in \cite{Kolbig}. More recently, we found a new compilation
of representations of generalized hypergeometric functions in \cite[Chap. 8]%
{Brychov}.

In existing literature, we found several papers devoted to the calculation
of hypergeometric summation formulas for a given number of parameters and
particular values of the argument (see i.e. \cite{Rakha}, \cite{Lewanowicz},
\cite{Kim}, \cite{Choi}, \cite{Atash}). However, the number of papers
devoted to the calculation of hypergeometric summation formulas for an
arbitrary number of parameters as well as an arbitrary argument is
relatively scarce (see i.e. \cite{Karlson}, \cite{pFqJL}). The aim of this
paper is to contribute to this last type of articles.

This paper is organized as follows. Section \ref{Section: Preliminaries}
presents the special functions that will be used throughout the paper, as
well as some of their basic properties. Section \ref{Section: Main results}
proves the main hypergeometric identity on which most of the results
obtained in the article are based. In Section \ref{Section: particular
arguments} we apply this hypergeometric identity to obtain reduction
formulas for particular arguments, while in Sections \ref{Section: finite p
and arbitrary z} and \ref{Section: arbitrary p and z} we apply it to
arbitrary arguments. Finally, we collect our conclusions in Section \ref%
{Section: conclusions}.

All the results presented in this paper have been tested with MATHEMATICA.
The corresponding MATHEMATICA notebook is available at %
\url{https://shorturl.at/tGOwb}.

\section{Preliminaries}

\label{Section: Preliminaries}

The gamma function is usually defined as \cite[Eqn. 1.1.1]{Lebedev}
\begin{equation}
\Gamma \left( z\right) :=\int_{0}^{\infty }e^{-t}t^{z-1}dt.\quad \mathrm{Re}%
z>0,  \label{gamma_def}
\end{equation}%
and satisfies the reflection formula \cite[Eqn. 1.2.2]{Lebedev}%
\begin{equation}
\Gamma \left( z\right) \,\Gamma \left( 1-z\right) =\frac{\pi }{\sin \pi z}%
,\quad z\notin
\mathbb{Z}
.  \label{gamma_reflection}
\end{equation}%
The logarithmic derivative of $\Gamma \left( z\right) $ is defined as \cite[%
Eqn. 1.3.1]{Lebedev}:%
\begin{equation}
\psi \left( z\right) :=\frac{\Gamma ^{\prime }\left( z\right) }{\Gamma
\left( z\right) },  \label{Psi_def}
\end{equation}%
and the Bateman's $G$ function is defined in terms of the $\psi \left(
z\right) $ function as \cite[Eqn. 44:13:1]{Atlas}:
\begin{equation}
\beta \left( z\right) :=\frac{1}{2}\left[ \psi \left( \frac{z+1}{2}\right)
-\psi \left( \frac{z}{2}\right) \right] .  \label{Bateman_G_def}
\end{equation}%
The beta function is given by \cite[Eqn. 43:13:1-2]{Atlas}%
\begin{equation}
\mathrm{B}\left( \alpha ,\beta \right) :=\int_{0}^{1}t^{\alpha -1}\left(
1-t\right) ^{\beta -1}dt=\frac{\Gamma \left( \alpha \right) \,\Gamma \left(
\beta \right) }{\Gamma \left( \alpha +\beta \right) },  \label{beta_def}
\end{equation}%
and the incomplete beta function \cite[Eqn. 8.17.1]{NIST}:%
\begin{equation}
\mathrm{B}_{z}\left( \alpha ,\beta \right) :=\int_{0}^{z}t^{\alpha -1}\left(
1-t\right) ^{\beta -1}dt.  \label{incomplete_beta_def}
\end{equation}%
The Pochhammer polynomial can be expressed in terms of gamma functions as%
\begin{equation}
\left( \alpha \right) _{k}:=\prod_{j=0}^{k-1}\left( \alpha +j\right) =\frac{%
\Gamma \left( \alpha +k\right) }{\Gamma \left( \alpha \right) }.
\label{Pochhammer_def}
\end{equation}%
Note that, according to \cite[Eqn. 18:5:1]{Atlas}%
\begin{equation}
\left( -\alpha \right) _{k}=\left( -1\right) ^{k}\left( \alpha -k+1\right)
_{k},  \label{(-x)_n}
\end{equation}%
thus for $m=0,1,2,\ldots $, we have
\begin{equation}
\left( -m\right) _{k}=\frac{\left( -1\right) ^{k}m!}{\left( m-k\right) !}.
\label{(-m)_k}
\end{equation}%
Also, applying the reflection formula (\ref{gamma_reflection}), we have%
\begin{equation}
\left( \alpha \right) _{-n}=\frac{\left( -1\right) ^{n}}{\left( 1-\alpha
\right) _{n}}.  \label{(alpha)_-n}
\end{equation}%
The generalized hypergeometric function is defined by the hypergeometric
series \cite[Eqn. 16.2.1]{NIST}
\begin{equation}
_{p}F_{q}\left( \left.
\begin{array}{c}
a_{1},\ldots ,a_{p} \\
b_{1},\ldots ,b_{q}%
\end{array}%
\right\vert z\right) =\sum_{k=0}^{\infty }\frac{\left( a_{1}\right)
_{k}\cdots \left( a_{p}\right) _{k}\,z^{k}}{k!\,\left( b_{1}\right)
_{k}\cdots \left( b_{q}\right) _{k}},  \label{Hypergeometric_sum}
\end{equation}%
wherever this series converges, and elsewhere by analytic continuation.
According to \cite[Eqn. 16.3.2]{NIST}, the following differentiation formula
is satisfied:%
\begin{eqnarray}
&&\frac{d^{n}}{dz^{n}}\left[ z^{\gamma }{}_{p}F_{q}\left( \left.
\begin{array}{c}
a_{1},\ldots ,a_{p} \\
b_{1},\ldots ,b_{q}%
\end{array}%
\right\vert z\right) \right]  \label{Diff_pFq_general} \\
&=&\left( \gamma -n+1\right) _{n}\,z^{\gamma -n}\,_{p+1}F_{q+1}\left( \left.
\begin{array}{c}
\gamma +1,a_{1},\ldots ,a_{p}, \\
\gamma +1-n,b_{1},\ldots ,b_{q},%
\end{array}%
\right\vert z\right) ,  \notag
\end{eqnarray}%
thus\
\begin{eqnarray}
&&_{p+1}F_{q+1}\left( \left.
\begin{array}{c}
a_{1},\ldots ,a_{p},a+n \\
b_{1},\ldots ,b_{q},a%
\end{array}%
\right\vert z\right)  \label{Diff_pFq} \\
&=&\frac{z^{1-a}}{\left( a\right) _{n}}\,\frac{d^{n}}{dz^{n}}\left[
z^{n+a-1}\,_{p}F_{q}\left( \left.
\begin{array}{c}
a_{1},\ldots ,a_{p} \\
b_{1},\ldots ,b_{q}%
\end{array}%
\right\vert z\right) \right] .  \notag
\end{eqnarray}%
Also, the Chu-Vandermonde summation formula is \cite[Eqn. 15.4.24]{NIST}:%
\begin{equation}
_{2}F_{1}\left( \left.
\begin{array}{c}
-n,a \\
c%
\end{array}%
\right\vert 1\right) =\frac{\left( c-a\right) _{n}}{\left( c\right) _{n}}.
\label{Chu_Vandermonde}
\end{equation}%
Further, according to \cite[Eqn. 15.8.1]{NIST}, we have the linear
transformation formula:%
\begin{equation}
_{2}F_{1}\left( \left.
\begin{array}{c}
a,b \\
c%
\end{array}%
\right\vert z\right) =\left( 1-z\right) ^{-a}\,_{2}F_{1}\left( \left.
\begin{array}{c}
a,c-b \\
c%
\end{array}%
\right\vert \frac{z}{z-1}\right) .  \label{2F1_linear}
\end{equation}%
Finally, Leibniz's differentiation formula \cite[Eqn. 1.4.2]{NIST} is given
by
\begin{equation}
\frac{d^{n}}{dz^{n}}\left[ f\left( z\right) \,g\left( z\right) \right]
=\sum_{k=0}^{n}\binom{n}{k}\frac{d^{k}}{dz^{k}}\left[ f\left( z\right) %
\right] \frac{d^{n-k}}{dz^{n-k}}\left[ g\left( z\right) \right] .
\label{Leibniz}
\end{equation}

\section{Main result}

\label{Section: Main results}

\begin{theorem}
\label{Theorem: Main}The following representation of the generalized
hypergeometric function $_{p}F_{q}$ holds true for $n=0,1,2,\ldots $%
\begin{eqnarray}
&&_{p}F_{q}\left( \left.
\begin{array}{c}
a_{1},\ldots ,a_{p} \\
b_{1},\ldots ,b_{q}%
\end{array}%
\right\vert z\right)  \label{Main_theorem} \\
&=&\sum_{k=0}^{n}\binom{n}{k}\frac{\,\left( a_{1}\right) _{k}\cdots \left(
a_{p}\right) _{k}\,\,z^{k}}{\,\left( b_{1}\right) _{k}\cdots \left(
b_{q}\right) _{k}\,\left( c+n\right) _{k}}\,_{p+1}F_{q+1}\left( \left.
\begin{array}{c}
a_{1}+k,\ldots ,\,a_{p}+k,\,c+k \\
b_{1}+k,\ldots ,\,b_{q}+k,\,c+n+k%
\end{array}%
\right\vert z\right) .  \notag
\end{eqnarray}
\end{theorem}

\begin{proof}
According to the Chu-Vandermonde summation formula (\ref{Chu_Vandermonde}),
the definition of the hypergeometric sum (\ref{Hypergeometric_sum}), and the
property (\ref{(-m)_k}), we have%
\begin{eqnarray}
1 &=&\frac{\left( c\right) _{j}}{\left( c+n\right) _{j}}\,_{2}F_{1}\left(
\left.
\begin{array}{c}
-n,-j \\
c%
\end{array}%
\right\vert 1\right) =\frac{\left( c\right) _{j}}{\left( c+n\right) _{j}}%
\sum_{k=0}^{n}\frac{\left( -n\right) _{k}\,\left( -j\right) _{k}}{k!\left(
c\right) _{k}}  \notag \\
&=&\frac{\left( c\right) _{j}}{\left( c+n\right) _{j}}\sum_{k=0}^{\infty }%
\binom{n}{k}\frac{\left( -1\right) ^{k}\left( -j\right) _{k}}{\left(
c\right) _{k}}.  \label{1=expression}
\end{eqnarray}%
Insert (\ref{1=expression})\ into (\ref{Hypergeometric_sum})\ and exchange
the summation order to arrive at%
\begin{eqnarray}
&&_{p}F_{q}\left( \left.
\begin{array}{c}
a_{1},\ldots ,a_{p} \\
b_{1},\ldots ,b_{q}%
\end{array}%
\right\vert z\right)  \notag \\
&=&\sum_{k=0}^{n}\binom{n}{k}\frac{\left( -1\right) ^{k}}{\left( c\right)
_{k}}\underset{S}{\underbrace{\sum_{j=k}^{\infty }\frac{\left( -j\right)
_{k}\,\left( a_{1}\right) _{j}\cdots \left( a_{p}\right) _{j}}{j!\,\left(
b_{1}\right) _{j}\cdots \left( b_{q}\right) _{j}}\frac{\left( c\right) _{j}}{%
\left( c+n\right) _{j}}z^{j}}},  \label{p_F_q_S}
\end{eqnarray}%
where the summation indices are $k=0,1,2,\ldots ,n$ and $j=k,k+1,\ldots $.
Now, perform the change $\ell =j-k$, thus
\begin{equation*}
S=\sum_{\ell =0}^{\infty }\frac{\left( -k-\ell \right) _{k}\,\left(
a_{1}\right) _{k+\ell }\cdots \left( a_{p}\right) _{k+\ell }}{\left( k+\ell
\right) !\,\left( b_{1}\right) _{k+\ell }\cdots \left( b_{q}\right) _{k+\ell
}}\frac{\left( c\right) _{k+\ell }}{\left( c+n\right) _{k+\ell }}z^{k+\ell }.
\end{equation*}%
From (\ref{(-x)_n}), we have%
\begin{equation}
\left( -k-\ell \right) _{k}=\left( -1\right) ^{k}\left( \ell +1\right) _{k},
\label{Prop_1}
\end{equation}%
and from (\ref{Pochhammer_def}),%
\begin{equation}
\left( a\right) _{k+\ell }=\left( a\right) _{k}\left( a+k\right) _{\ell },
\label{Prop_2}
\end{equation}%
and%
\begin{equation}
\left( k+\ell \right) !=\ell !\left( 1+\ell \right) _{k}.  \label{Prop_3}
\end{equation}%
Take into account (\ref{Prop_1})-(\ref{Prop_3}), as well as the definition (%
\ref{Hypergeometric_sum}), to obtain
\begin{eqnarray}
&&S  \label{S_resultado} \\
= &&\frac{\,\left( a_{1}\right) _{k}\cdots \left( a_{p}\right) _{k}}{\left(
b_{1}\right) _{k}\cdots \left( b_{q}\right) _{k}}\frac{\left( c\right)
_{k}\,\left( -z\right) ^{k}}{\left( c+n\right) _{k}}\sum_{\ell =0}^{\infty }%
\frac{\,\left( a_{1}+k\right) _{\ell }\cdots \left( a_{p}+k\right) _{\ell }}{%
\ell !\,\left( b_{1}+k\right) _{\ell }\cdots \left( b_{q}\right) _{k+\ell }}%
\frac{\left( c+k\right) _{\ell }}{\left( c+n+k\right) _{\ell }}z^{\ell }
\notag \\
= &&\frac{\,\left( a_{1}\right) _{k}\cdots \left( a_{p}\right) _{k}}{\left(
b_{1}\right) _{k}\cdots \left( b_{q}\right) _{k}}\frac{\left( c\right)
_{k}\,\left( -z\right) ^{k}}{\left( c+n\right) _{k}}\,_{p+1}F_{q+1}\left(
\left.
\begin{array}{c}
a_{1}+k,\ldots ,\,a_{p}+k,\,c+k \\
b_{1}+k,\ldots ,\,b_{q}+k,\,c+n+k%
\end{array}%
\right\vert z\right) .  \notag
\end{eqnarray}%
Finally, substitute (\ref{S_resultado})\ in (\ref{p_F_q_S}) to complete the
proof.
\end{proof}

\begin{corollary}
\label{Corollary: main}For $n=0,1,2,\ldots $the following reduction formula
holds true:\
\begin{eqnarray}
&&_{p+1}F_{q+1}\left( \left.
\begin{array}{c}
a_{1},\ldots ,a_{p},c+n \\
b_{1},\ldots ,b_{q},c%
\end{array}%
\right\vert z\right)   \label{Corollary_main} \\
&=&\sum_{k=0}^{n}\binom{n}{k}\frac{\,\left( a_{1}\right) _{k}\cdots \left(
a_{p}\right) _{k}\,\,z^{k}}{\,\left( b_{1}\right) _{k}\cdots \left(
b_{q}\right) _{k}\,\left( c\right) _{k}}\,_{p}F_{q}\left( \left.
\begin{array}{c}
a_{1}+k,\ldots ,\,a_{p}+k \\
b_{1}+k,\ldots ,\,b_{q}+k%
\end{array}%
\right\vert z\right) .  \notag
\end{eqnarray}
\end{corollary}

\section{Application to reduction formulas with arguments $z=\frac{1}{2},1$}

\label{Section: particular arguments}\bigskip

\begin{theorem}
For $a,c\in
\mathbb{C}
$ and $n=0,1,2,\ldots $, the following reduction formula holds true:%
\begin{equation}
_{3}F_{2}\left( \left.
\begin{array}{c}
a,a,c+n \\
a+1,c%
\end{array}%
\right\vert \frac{1}{2}\right) =a\,2^{a}\sum_{k=0}^{n}\binom{n}{k}\frac{%
\left( a\right) _{k}}{\left( c\right) _{k}}\,\beta \left( a+k\right) .
\label{3F2_(1/2)_0}
\end{equation}
\end{theorem}

\begin{proof}
According to (\ref{Corollary: main}), we have%
\begin{eqnarray*}
&&_{3}F_{2}\left( \left.
\begin{array}{c}
a,a,c+n \\
a+1,c%
\end{array}%
\right\vert \frac{1}{2}\right) \\
&=&\sum_{k=0}^{n}\binom{n}{k}\frac{\left[ \left( a\right) _{k}\right]
^{2}\,2^{-k}}{\left( a+1\right) _{k}\,\left( c\right) _{k}}\,_{2}F_{1}\left(
\left.
\begin{array}{c}
a+k,a+k \\
a+1+k%
\end{array}%
\right\vert \frac{1}{2}\right) .
\end{eqnarray*}%
Apply the identity \
\begin{equation}
\frac{\left( a\right) _{k}}{\left( a+1\right) _{k}}=\frac{a}{a+k},
\label{(a)_k/(a+1)_k}
\end{equation}%
\ and the reduction formula \cite[Eqn. 7.3.7(16)]{Prudnikov3}:
\begin{equation*}
_{2}F_{1}\left( \left.
\begin{array}{c}
a,a \\
a+1%
\end{array}%
\right\vert \frac{1}{2}\right) =2^{a}a\,\,\beta \left( a\right) ,
\end{equation*}%
to complete the proof.
\end{proof}

\begin{remark}
For $c=a$, (\ref{3F2_(1/2)_0})\ reduces to%
\begin{equation*}
_{2}F_{1}\left( \left.
\begin{array}{c}
a,a+n \\
a+1%
\end{array}%
\right\vert \frac{1}{2}\right) =a\,2^{a}\sum_{k=0}^{n}\binom{n}{k}\,\beta
\left( a+k\right) .
\end{equation*}%
Applying the linear transformation formula (\ref{2F1_linear}) and the change
$n\rightarrow n+1$, we obtain%
\begin{equation}
_{2}F_{1}\left( \left.
\begin{array}{c}
-n,a \\
a+1%
\end{array}%
\right\vert -1\right) =a\sum_{k=0}^{n+1}\binom{n+1}{k}\,\beta \left(
a+k\right) .  \label{2F1(-1)_a}
\end{equation}%
However, applying (\ref{(-m)_k})\ and (\ref{(a)_k/(a+1)_k}) to the
definition of the hypergeometric function (\ref{Hypergeometric_sum}), we
arrive at
\begin{equation}
_{2}F_{1}\left( \left.
\begin{array}{c}
-n,a \\
a+1%
\end{array}%
\right\vert -1\right) =a\sum_{k=0}^{n}\binom{n}{k}\frac{1}{a+k}.
\label{2F1(-1)_b}
\end{equation}%
From (\ref{2F1(-1)_a})\ and (\ref{2F1(-1)_b}), we obtain the following
recursive equation for the Bateman's $G$ function:%
\begin{equation}
\beta \left( a+n+1\right) =\sum_{k=0}^{n}\binom{n}{k}\left[ \frac{1}{a+k}-%
\frac{n+1}{n+1-k}\beta \left( a+k\right) \right] .  \label{Beta_recursive}
\end{equation}
\end{remark}

\begin{theorem}
For $a,b,c\in
\mathbb{C}
$ and $n,m=0,1,2,\ldots $, the following reduction formula holds true:%
\begin{eqnarray}
&&_{3}F_{2}\left( \left.
\begin{array}{c}
a,b+m,c+n \\
\frac{a+b+1}{2},c%
\end{array}%
\right\vert \frac{1}{2}\right)  \label{3F2_(1/2)_A} \\
&=&\frac{2^{a-1}\,\Gamma \left( \frac{a+b+1}{2}\right) }{\Gamma \left(
a\right) }\sum_{k=0}^{n}\binom{n}{k}\frac{\left( b+m\right) _{k}}{\left(
c\right) _{k}}\sum_{s=0}^{m}\binom{m}{s}\frac{\Gamma \left( \frac{a+k+s}{2}%
\right) }{\Gamma \left( \frac{1+b+k+s}{2}\right) }.  \notag
\end{eqnarray}
\end{theorem}

\begin{proof}
According to (\ref{Corollary_main}), we have%
\begin{eqnarray*}
&&_{3}F_{2}\left( \left.
\begin{array}{c}
a,b+m,c+n \\
\frac{a+b+1}{2},c%
\end{array}%
\right\vert \frac{1}{2}\right) \\
&=&\sum_{k=0}^{n}\binom{n}{k}\frac{\left( a\right) _{k}\left( b+m\right) _{k}%
}{\left( \frac{a+b+1}{2}\right) _{k}\left( c\right) _{k}}\,2^{-k}\,_{2}F_{1}%
\left( \left.
\begin{array}{c}
a+k,b+m+k \\
\frac{a+b+1}{2}+k%
\end{array}%
\right\vert \frac{1}{2}\right) .
\end{eqnarray*}%
Apply the reduction formula \cite[7.3.7(2)]{Prudnikov3}:%
\begin{equation*}
_{2}F_{1}\left( \left.
\begin{array}{c}
a,b+m \\
\frac{a+b+1}{2}%
\end{array}%
\right\vert \frac{1}{2}\right) =\frac{2^{a-1}\,\Gamma \left( \frac{a+b+1}{2}%
\right) }{\Gamma \left( a\right) }\sum_{s=0}^{m}\binom{m}{s}\frac{\Gamma
\left( \frac{a+s}{2}\right) }{\Gamma \left( \frac{1+b+s}{2}\right) },
\end{equation*}%
and simplify the result to complete the proof.
\end{proof}

\begin{theorem}
For $a,b,c\in
\mathbb{C}
$ and $n,m=0,1,2,\ldots $, the following reduction formula holds true:%
\begin{eqnarray}
&&_{3}F_{2}\left( \left.
\begin{array}{c}
a,b-m,c+n \\
\frac{a+b+1}{2},c%
\end{array}%
\right\vert \frac{1}{2}\right)  \label{3F2_(1/2)_B} \\
&=&\frac{2^{a-1}\,\Gamma \left( \frac{a+b+1}{2}\right) \,\Gamma \left( \frac{%
b-a+1}{2}-m\right) }{\Gamma \left( a\right) \,\Gamma \left( \frac{b-a+1}{2}%
\right) }\sum_{k=0}^{n}\binom{n}{k}\frac{\left( b-m\right) _{k}}{\left(
c\right) _{k}}\sum_{s=0}^{m}\binom{m}{s}\frac{\Gamma \left( \frac{a+k+s}{2}%
\right) }{\Gamma \left( \frac{1+b+k+s}{2}-m\right) }.  \notag
\end{eqnarray}
\end{theorem}

\begin{proof}
According to (\ref{Corollary_main}), we have%
\begin{eqnarray*}
&&_{3}F_{2}\left( \left.
\begin{array}{c}
a,b-m,c+n \\
\frac{a+b+1}{2},c%
\end{array}%
\right\vert \frac{1}{2}\right) \\
&=&\sum_{k=0}^{n}\binom{n}{k}\frac{\left( a\right) _{k}\left( b-m\right) _{k}%
}{\left( \frac{a+b+1}{2}\right) _{k}\left( c\right) _{k}}\,2^{-k}\,_{2}F_{1}%
\left( \left.
\begin{array}{c}
a+k,b-m+k \\
\frac{a+b+1}{2}+k%
\end{array}%
\right\vert \frac{1}{2}\right) .
\end{eqnarray*}%
Apply the reduction formula \cite{Rakha}:%
\begin{equation*}
_{2}F_{1}\left( \left.
\begin{array}{c}
a,b-m \\
\frac{a+b+1}{2}%
\end{array}%
\right\vert \frac{1}{2}\right) =\frac{2^{a-1}\,\Gamma \left( \frac{a+b+1}{2}%
\right) \,\Gamma \left( \frac{b-a+1}{2}-m\right) }{\Gamma \left( a\right)
\,\Gamma \left( \frac{b-a+1}{2}-m\right) }\sum_{s=0}^{m}\binom{m}{s}\frac{%
\left( -1\right) ^{s}\,\Gamma \left( \frac{a+s}{2}\right) }{\Gamma \left(
\frac{1+b+s}{2}-m\right) },
\end{equation*}%
and simplify the result to complete the proof.
\end{proof}

\bigskip

For the next result, we need to prove first the following Lemma.

\begin{lemma}
For $b\neq c$, the following finite sum holds true:%
\begin{eqnarray}
&&\sum_{k=0}^{n}\binom{n}{k}\left( -1\right) ^{k}\frac{\,\left( b\right) _{k}%
}{\left( c\right) _{k}}\psi \left( b+k\right)  \label{Sum_Psi} \\
&=&\frac{\left( c-b\right) _{n}}{\left( c\right) _{n}}\left[ \psi \left(
1+b-c\right) +\psi \left( b\right) -\psi \left( 1+b-c-n\right) \right] .
\notag
\end{eqnarray}
\end{lemma}

\begin{proof}
Apply (\ref{Corollary: main}), to obtain%
\begin{eqnarray*}
&&_{4}F_{3}\left( \left.
\begin{array}{c}
a,b,b,c+n \\
b+1,b+1,c%
\end{array}%
\right\vert 1\right) \\
&=&\sum_{k=0}^{n}\binom{n}{k}\frac{\left( a\right) _{k}}{\left( c\right) _{k}%
}\left( \frac{\left( b\right) _{k}}{\left( b+1\right) _{k}}\right)
^{2}\,_{3}F_{2}\left( \left.
\begin{array}{c}
a+k,b+k,b+k \\
b+k+1,b+k+1%
\end{array}%
\right\vert 1\right) .
\end{eqnarray*}%
Taking into account (\ref{(a)_k/(a+1)_k}) and the reduction formula \cite%
{SumsJL},%
\begin{equation}
_{3}F_{2}\left( \left.
\begin{array}{c}
a,b,b \\
b+1,b+1%
\end{array}%
\right\vert 1\right) =b^{2}\,\mathrm{B}\left( 1-a,b\right) \left[ \psi
\left( 1+b-a\right) -\psi \left( b\right) \right] ,  \label{3F2_JL}
\end{equation}
after simplification, we arrive at%
\begin{eqnarray*}
&&_{4}F_{3}\left( \left.
\begin{array}{c}
a,b,b,c+n \\
b+1,b+1,c%
\end{array}%
\right\vert 1\right) \\
&=&b^{2}\,\mathrm{B}\left( 1-a,b\right) \sum_{k=0}^{n}\binom{n}{k}\frac{%
\left( a\right) _{k}\,\left( b\right) _{k}\left( 1-a\right) _{-k}}{\left(
c\right) _{k}}\left[ \psi \left( 1+b-a\right) -\psi \left( b+k\right) \right]
.
\end{eqnarray*}%
Now, applying (\ref{(alpha)_-n}), we obtain
\begin{eqnarray}
&&_{4}F_{3}\left( \left.
\begin{array}{c}
a,b,b,c+n \\
b+1,b+1,c%
\end{array}%
\right\vert 1\right) =b^{2}\,\mathrm{B}\left( 1-a,b\right)  \notag \\
&&\left[ \psi \left( 1+b-a\right) \sum_{k=0}^{n}\binom{n}{k}\left( -1\right)
^{k}\frac{\,\left( b\right) _{k}}{\left( c\right) _{k}}-\sum_{k=0}^{n}\binom{%
n}{k}\left( -1\right) ^{k}\frac{\,\left( b\right) _{k}}{\left( c\right) _{k}}%
\psi \left( b+k\right) \right]  \label{4F3_insert}
\end{eqnarray}%
However, according to (\ref{(-m)_k})\ and the Chu-Vandermonde summation
formula (\ref{Chu_Vandermonde}), we have
\begin{equation}
\sum_{k=0}^{n}\binom{n}{k}\left( -1\right) ^{k}\frac{\,\left( b\right) _{k}}{%
\left( c\right) _{k}}=\,_{2}F_{1}\left( \left.
\begin{array}{c}
-n,b \\
c%
\end{array}%
\right\vert 1\right) =\frac{\left( c-b\right) _{n}}{\left( c\right) _{n}},
\label{Chu_4F3}
\end{equation}%
hence, inserting (\ref{Chu_4F3})\ into (\ref{4F3_insert}) we obtain%
\begin{eqnarray}
&&_{4}F_{3}\left( \left.
\begin{array}{c}
a,b,b,c+n \\
b+1,b+1,c%
\end{array}%
\right\vert 1\right)  \label{4F3_sum} \\
&=&b^{2}\,\mathrm{B}\left( 1-a,b\right) \left[ \psi \left( 1+b-a\right)
\frac{\left( c-b\right) _{n}}{\left( c\right) _{n}}-\sum_{k=0}^{n}\binom{n}{k%
}\left( -1\right) ^{k}\frac{\left( b\right) _{k}}{\left( c\right) _{k}}\psi
\left( b+k\right) \right] .  \notag
\end{eqnarray}%
Notice that for $a=c$, (\ref{4F3_sum})\ reduces to%
\begin{eqnarray}
&&_{3}F_{2}\left( \left.
\begin{array}{c}
b,b,c+n \\
b+1,b+1%
\end{array}%
\right\vert 1\right)  \label{3F2_a} \\
&=&b^{2}\,\mathrm{B}\left( 1-c,b\right) \left[ \psi \left( 1+b-c\right)
\frac{\left( c-b\right) _{n}}{\left( c\right) _{n}}-\sum_{k=0}^{n}\binom{n}{k%
}\left( -1\right) ^{k}\frac{\left( b\right) _{k}}{\left( c\right) _{k}}\psi
\left( b+k\right) \right] ,  \notag
\end{eqnarray}%
but, according to (\ref{3F2_JL}), we have
\begin{equation}
_{3}F_{2}\left( \left.
\begin{array}{c}
b,b,c+n \\
b+1,b+1%
\end{array}%
\right\vert 1\right) =b^{2}\,\mathrm{B}\left( 1-c-n,b\right) \left[ \psi
\left( 1+b-c-n\right) -\psi \left( b\right) \right] .  \label{3F2_b}
\end{equation}%
Compare (\ref{3F2_a})\ with (\ref{3F2_b})\ to obtain%
\begin{eqnarray*}
&&\sum_{k=0}^{n}\binom{n}{k}\left( -1\right) ^{k}\frac{\,\left( b\right) _{k}%
}{\left( c\right) _{k}}\psi \left( b+k\right) \\
&=&\frac{\mathrm{B}\left( 1-c-n,b\right) }{\mathrm{B}\left( 1-c,b\right) }%
\left[ \psi \left( b\right) -\psi \left( 1+b-c-n\right) \right] +\psi \left(
1+b-c\right) \frac{\left( c-b\right) _{n}}{\left( c\right) _{n}}.
\end{eqnarray*}%
Finally, complete the proof applying the identity%
\begin{equation*}
\frac{\mathrm{B}\left( 1-c-n,b\right) }{\mathrm{B}\left( 1-c,b\right) }=%
\frac{\left( 1-c\right) _{-n}}{\left( 1-c+b\right) _{-n}}=\frac{\left(
c-b\right) _{n}}{\left( c\right) _{n}}.
\end{equation*}
\end{proof}

\begin{theorem}
For $a,b,c\in
\mathbb{C}
$ with $b\neq c$, and $n=0,1,2,\ldots $, the following reduction formula
holds true:%
\begin{eqnarray}
&&_{4}F_{3}\left( \left.
\begin{array}{c}
a,b,b,c+n \\
b+1,b+1,c%
\end{array}%
\right\vert 1\right) =b^{2}\,\mathrm{B}\left( 1-a,b\right)
\label{4F3(1)_reduction} \\
&&\frac{\left( c-b\right) _{n}}{\left( c\right) _{n}}\left[ \psi \left(
1+b-a\right) -\psi \left( 1+b-c\right) -\psi \left( b\right) +\psi \left(
1+b-c-n\right) \right] .  \notag
\end{eqnarray}
\end{theorem}

\begin{proof}
Insert (\ref{Sum_Psi})\ into (\ref{4F3_sum})\ to arrive at the desired
result.
\end{proof}

\bigskip

Note that the case $b=c$ is not included in (\ref{Sum_Psi}). Next, we derive
this case.

\bigskip

\begin{lemma}
For $n=1,2,\ldots $ the following finite sum holds true:%
\begin{equation}
\sum_{k=0}^{n}\binom{n}{k}\left( -1\right) ^{k}\,\psi \left( b+k\right) =-%
\frac{\left( n-1\right) !}{\left( b\right) _{n}}.  \label{Sum_Psi_simple}
\end{equation}
\end{lemma}

\begin{proof}
Take $c=b$ in (\ref{4F3(1)_reduction})\ to obtain
\begin{equation}
_{3}F_{2}\left( \left.
\begin{array}{c}
a,b,b+n \\
b+1,b+1%
\end{array}%
\right\vert 1\right) =-b^{2}\,\mathrm{B}\left( 1-a,b\right) \sum_{k=0}^{n}%
\binom{n}{k}\left( -1\right) ^{k}\,\psi \left( b+k\right) .
\label{3F2_simple_sum}
\end{equation}%
Further, take $a=b$, thus%
\begin{equation}
_{3}F_{2}\left( \left.
\begin{array}{c}
b,b,b+n \\
b+1,b+1%
\end{array}%
\right\vert 1\right) =-b^{2}\,\mathrm{B}\left( 1-b,b\right) \sum_{k=0}^{n}%
\binom{n}{k}\left( -1\right) ^{k}\,\psi \left( b+k\right) .
\label{3F2_simple_a}
\end{equation}%
However, from (\ref{3F2_JL}), we have%
\begin{equation}
_{3}F_{2}\left( \left.
\begin{array}{c}
b,b,b+n \\
b+1,b+1%
\end{array}%
\right\vert 1\right) =b^{2}\,\mathrm{B}\left( 1-b-n,b\right) \left[ \psi
\left( 1-n\right) -\psi \left( b\right) \right] .  \label{3F2_simple_b}
\end{equation}%
Compare (\ref{3F2_simple_a})\ with (\ref{3F2_simple_b})\ and take into
account (\ref{beta_def}) and (\ref{(alpha)_-n}) to arrive at%
\begin{equation*}
\sum_{k=0}^{n}\binom{n}{k}\left( -1\right) ^{k}\,\psi \left( b+k\right) =-%
\frac{\left( -1\right) ^{n}}{\left( b\right) _{n}}\left[ \frac{\psi \left(
1-n\right) }{\Gamma \left( 1-n\right) }-\frac{\psi \left( b\right) }{\Gamma
\left( 1-n\right) }\right] .
\end{equation*}%
Finally, for $m=0,1,2,\ldots $, apply \cite[Eqn. 1.1.5\&1.3.2]{Lebedev}
\begin{eqnarray*}
\lim_{z\rightarrow -m}\Gamma \left( z\right) &=&\frac{\left( -1\right) ^{m}}{%
m!}\lim_{z\rightarrow -m}\frac{1}{z+m}, \\
\lim_{z\rightarrow -m}\psi \left( z\right) &=&-\lim_{z\rightarrow -m}\frac{1%
}{z+m},
\end{eqnarray*}%
to complete the proof.
\end{proof}

\begin{corollary}
Insert (\ref{Sum_Psi_simple})\ into (\ref{3F2_simple_sum})\ to obtain, for $%
n=1,2,\ldots $%
\begin{equation}
_{3}F_{2}\left( \left.
\begin{array}{c}
a,b,b+n \\
b+1,b+1%
\end{array}%
\right\vert 1\right) =\left( n-1\right) !\frac{b^{2}\,\mathrm{B}\left(
1-a,b\right) }{\left( b\right) _{n}}.  \label{3F2(1)_corollary}
\end{equation}
\end{corollary}

\begin{theorem}
For $n=1,2,\ldots $ and $m=0,1,2,\ldots $The following reduction formula
holds true:%
\begin{equation}
_{4}F_{3}\left( \left.
\begin{array}{c}
a,b,b+n,c+m \\
b+1,b+1,c%
\end{array}%
\right\vert 1\right) =\left( n-1\right) !\frac{\,b^{2}\mathrm{B}\left(
1-a,b\right) }{\left( b\right) _{n}}\frac{\left( c-b\right) _{m}}{\left(
c\right) _{m}}.  \label{4F3(1)_theorem}
\end{equation}
\end{theorem}

\begin{proof}
Application of (\ref{Corollary_main}) leads to%
\begin{eqnarray*}
&&_{4}F_{3}\left( \left.
\begin{array}{c}
a,b,b+n,c+m \\
b+1,b+1,c%
\end{array}%
\right\vert 1\right) \\
&=&\sum_{k=0}^{m}\binom{m}{k}\frac{\left( a\right) _{k}\left( b\right)
_{k}\left( b+n\right) _{k}}{\left[ \left( b+1\right) _{k}\right] ^{2}\left(
c\right) _{k}}\,_{3}F_{2}\left( \left.
\begin{array}{c}
a+k,b+k,b+n+k \\
b+1+k,b+1+k%
\end{array}%
\right\vert 1\right) .
\end{eqnarray*}%
Take into account (\ref{3F2(1)_corollary}) and simplify the result to arrive
at%
\begin{eqnarray}
&&_{4}F_{3}\left( \left.
\begin{array}{c}
a,b,b+n,c+m \\
b+1,b+1,c%
\end{array}%
\right\vert 1\right)  \label{4F3(1)_intermedio} \\
&=&\left( n-1\right) !\frac{b\,\Gamma \left( b+1\right) }{\Gamma \left(
1-a+b\right) }\sum_{k=0}^{m}\binom{m}{k}\frac{\left( a\right) _{k}\left(
b+n\right) _{k}}{\left( c\right) _{k}\left( b+k\right) _{n}}\,\Gamma \left(
1-a-k\right) .  \notag
\end{eqnarray}%
Note that, according to (\ref{Pochhammer_def}),%
\begin{equation*}
\frac{\left( b+n\right) _{k}}{\left( b+k\right) _{n}}=\frac{\left( b\right)
_{k}}{\left( b\right) _{n}},
\end{equation*}%
and according to (\ref{gamma_reflection})%
\begin{equation*}
\frac{\Gamma \left( 1-a-k\right) \,\Gamma \left( a+k\right) }{\Gamma \left(
a\right) \,\Gamma \left( 1-a\right) }=\left( -1\right) ^{k},
\end{equation*}%
thus (\ref{4F3(1)_intermedio})\ reduces to%
\begin{eqnarray*}
&&_{4}F_{3}\left( \left.
\begin{array}{c}
a,b,b+n,c+m \\
b+1,b+1,c%
\end{array}%
\right\vert 1\right) \\
&=&\left( n-1\right) !\frac{\,b^{2}\mathrm{B}\left( 1-a,b\right) }{\left(
b\right) _{n}}\sum_{k=0}^{m}\binom{m}{k}\left( -1\right) ^{k}\frac{\left(
a\right) _{k}}{\left( c\right) _{k}}.
\end{eqnarray*}%
Finally, apply (\ref{Chu_4F3})\ to complete the proof.
\end{proof}

\section{Application to $_{p}F_{p+1}$ and $_{p}F_{p}$ reduction formulas
with $p=1,2,3$ and arbitrary $z$}

\label{Section: finite p and arbitrary z}

\begin{theorem}
For $n=0,1,2,\ldots $ the following reduction formula holds true:%
\begin{equation}
_{1}F_{2}\left( \left.
\begin{array}{c}
c+n \\
b,c%
\end{array}%
\right\vert z\right) =z^{\left( 1-b\right) /2}\,\Gamma \left( b\right)
\sum_{k=0}^{n}\binom{n}{k}\frac{z^{k/2}}{\left( c\right) _{k}}%
\,I_{b+k-1}\left( 2\sqrt{z}\right) ,  \label{1F2_theorem}
\end{equation}%
where $I_{\nu }\left( z\right) $ denotes the modified Bessel function of the
first kind \cite[Eqn. 5.7.1]{Lebedev}.
\end{theorem}

\begin{proof}
According to (\ref{Corollary_main}), we have%
\begin{equation*}
_{1}F_{2}\left( \left.
\begin{array}{c}
c+n \\
b,c%
\end{array}%
\right\vert z\right) =\sum_{k=0}^{n}\binom{n}{k}\frac{z^{k}}{\left( b\right)
_{k}\left( c\right) _{k}}\,_{0}F_{1}\left( \left.
\begin{array}{c}
- \\
b+k%
\end{array}%
\right\vert z\right) .
\end{equation*}%
Apply the reduction formula \cite[Eqn. 7.13.1(1)]{Prudnikov3}:
\begin{equation}
_{0}F_{1}\left( \left.
\begin{array}{c}
- \\
b%
\end{array}%
\right\vert z\right) =z^{\left( 1-b\right) /2}\,\Gamma \left( b\right)
\,I_{b-1}\left( 2\sqrt{z}\right) ,  \label{0F1_reduction}
\end{equation}%
and simplify the result to complete the proof.
\end{proof}

\bigskip

We obtain an alternative proof of (\ref{1F2_theorem})\ as follows.

\bigskip

\begin{proof}
From (\ref{0F1_reduction}), we have%
\begin{equation*}
\frac{d^{n}}{dz^{n}}\left[ z^{c-1}\,_{0}F_{1}\left( \left.
\begin{array}{c}
- \\
b%
\end{array}%
\right\vert z\right) \right] =\Gamma \left( b\right) \,\frac{d^{n}}{dz^{n}}%
\left[ z^{c-1}\,z^{\left( 1-b\right) /2}\,I_{b-1}\left( 2\sqrt{z}\right) %
\right] .
\end{equation*}%
Apply Leibniz's differentiation formula (\ref{Leibniz}), and the
hypergeometric differentiation formula (\ref{Diff_pFq_general}),\ as well as
\cite[Eqn. 1.13.1(5)]{Brychov}:
\begin{equation*}
\frac{d^{n}}{dz^{n}}\left[ z^{\pm \nu /2}\,I_{\nu }\left( a\sqrt{z}\right) %
\right] =\left( \frac{a}{2}\right) ^{n}z^{\left( \pm \nu -n\right)
/2}\,I_{\nu \mp n}\left( a\sqrt{z}\right) ,
\end{equation*}%
to arrive at%
\begin{eqnarray*}
&&_{1}F_{2}\left( \left.
\begin{array}{c}
c \\
b,c-n%
\end{array}%
\right\vert z\right) \\
&=&\left( -1\right) ^{n}\frac{\Gamma \left( b\right) \,z^{\left( 1-b\right)
/2}}{\left( c-n\right) _{n}}\sum_{k=0}^{n}\binom{n}{k}\left( -1\right)
^{k}\left( 1-c\right) _{n-k}\,z^{k/2}\,I_{b+k-1}\left( 2\sqrt{z}\right) .
\end{eqnarray*}%
Perform the change of variables $c\rightarrow c+n$ and take into account the
identity:%
\begin{equation*}
\left( -1\right) ^{n+k}\frac{\left( 1-c-n\right) _{n-k}}{\left( c\right) _{n}%
}=\frac{1}{\left( c\right) _{k}},
\end{equation*}%
to obtain (\ref{1F2_theorem}), as we wanted to prove.
\end{proof}

\begin{theorem}
For $n,m=0,1,2,\ldots $the following reduction formula holds true:%
\begin{eqnarray}
&&_{2}F_{3}\left( \left.
\begin{array}{c}
c+n,d+m \\
b,c,d%
\end{array}%
\right\vert z\right)  \label{2F3_theorem} \\
&=&z^{\left( 1-b\right) /2}\Gamma \left( b\right) \,\Gamma \left( d\right)
\sum_{k=0}^{n}\binom{n}{k}\frac{\left( d+m\right) _{k}\,}{\left( c\right)
_{k}}\sum_{\ell =0}^{m}\binom{m}{\ell }\frac{z^{\left( k+\ell \right) /2}}{%
\Gamma \left( d+k+\ell \right) }\,I_{b+k+\ell -1}\left( 2\sqrt{z}\right) .
\notag
\end{eqnarray}
\end{theorem}

\begin{proof}
According to (\ref{Corollary_main}), we have
\begin{eqnarray*}
&&_{2}F_{3}\left( \left.
\begin{array}{c}
c+n,d+m \\
b,c,d%
\end{array}%
\right\vert z\right) \\
&=&\sum_{k=0}^{n}\binom{n}{k}\frac{\left( d+m\right) _{k}\,z^{k}}{\left(
b\right) _{k}\left( c\right) _{k}\left( d\right) _{k}}\,_{1}F_{2}\left(
\left.
\begin{array}{c}
d+m+k \\
b+k,d+k%
\end{array}%
\right\vert z\right) .
\end{eqnarray*}%
Apply (\ref{1F2_theorem})\ and simplify the result to complete the proof.
\end{proof}

\begin{remark}
The reduction formulas (\ref{1F2_theorem})\ and (\ref{2F3_theorem})\ can be
rewritten in terms of the Bessel function of the first find $J_{\nu }\left(
z\right) $ taking into account the reduction formula \cite[Eqn. 7.13.1(1)]%
{Prudnikov3}:%
\begin{equation*}
_{0}F_{1}\left( \left.
\begin{array}{c}
- \\
b%
\end{array}%
\right\vert -z\right) =z^{\left( 1-b\right) /2}\,\Gamma \left( b\right)
\,J_{b-1}\left( 2\sqrt{z}\right) .
\end{equation*}
\end{remark}

\begin{theorem}
For $n=0,1,2,\ldots $the following reduction formula holds true:%
\begin{equation*}
_{2}F_{2}\left( \left.
\begin{array}{c}
a,c+n \\
a+1,c%
\end{array}%
\right\vert z\right) =a\left( -z\right) ^{a}\sum_{k=0}^{n}\binom{n}{k}\left(
-1\right) ^{k}\,\frac{\gamma \left( a+k,-z\right) }{\left( c\right) _{k}},
\end{equation*}%
where $\gamma \left( \nu ,z\right) $ denotes the incomplete gamma function
\cite[Eqn. 45:3:1]{Atlas}.
\end{theorem}

\begin{proof}
According to (\ref{Corollary_main}), we have%
\begin{equation*}
_{2}F_{2}\left( \left.
\begin{array}{c}
a,c+n \\
a+1,c%
\end{array}%
\right\vert z\right) =\sum_{k=0}^{n}\binom{n}{k}\frac{\,\left( a\right)
_{k}\,\,z^{k}}{\,\left( a+1\right) _{k}\left( c\right) _{k}\,}%
\,_{1}F_{1}\left( \left.
\begin{array}{c}
a+k, \\
a+1+k,b+k%
\end{array}%
\right\vert z\right) .
\end{equation*}%
Apply (\ref{(a)_k/(a+1)_k}) and the reduction formula \cite[Eqn. 7.11.3(1)]%
{Prudnikov3}:%
\begin{equation*}
_{1}F_{1}\left( \left.
\begin{array}{c}
a, \\
a+1%
\end{array}%
\right\vert -z\right) =a\,z^{-a}\,\gamma \left( a,z\right) ,
\end{equation*}%
to complete the proof.
\end{proof}

\begin{theorem}
\label{Theorem: 2F2}For $n,m=0,1,2,\ldots $ the following reduction formula
holds true:
\begin{equation}
_{2}F_{2}\left( \left.
\begin{array}{c}
a+n,b+m \\
a,b%
\end{array}%
\right\vert z\right) =\frac{m!\,e^{z}}{\left( b\right) _{m}}\sum_{k=0}^{n}%
\binom{n}{k}\frac{z^{k}\,L_{m}^{b+k-1}\left( -z\right) }{\left( a\right) _{k}%
},  \label{2F2_theorem}
\end{equation}%
where $L_{n}^{\alpha }\left( z\right) $ denotes the generalized Laguerre
polynomial \cite[Sect. 4.17]{Lebedev}.
\end{theorem}

\begin{proof}
Performing the change $a\rightarrow a+n$ and taking into account (\ref%
{(-x)_n}), transform the reduction formula \cite[Eqn. 7.11.1(8)]{Prudnikov3}%
\begin{equation*}
_{1}F_{1}\left( \left.
\begin{array}{c}
a \\
a-n%
\end{array}%
\right\vert z\right) =\frac{\left( -1\right) ^{n}n!}{\left( 1-a\right) _{n}}%
\,e^{z}\,L_{n}^{a-n-1}\left( -z\right) ,
\end{equation*}%
into%
\begin{equation}
_{1}F_{1}\left( \left.
\begin{array}{c}
a+n \\
a%
\end{array}%
\right\vert z\right) =\frac{n!}{\left( a\right) _{n}}\,e^{z}\,L_{n}^{a-1}%
\left( -z\right) .  \label{1F1_Laguerre}
\end{equation}%
According to (\ref{Corollary: main}), we have
\begin{equation*}
_{2}F_{2}\left( \left.
\begin{array}{c}
a+n,b+m \\
a,b%
\end{array}%
\right\vert z\right) =\sum_{k=0}^{n}\binom{n}{k}\frac{\left( a+n\right)
_{k}\,z^{k}}{\left( a\right) _{k}\left( b\right) _{k}}\,_{1}F_{1}\left(
\left.
\begin{array}{c}
a+n+k \\
a+k%
\end{array}%
\right\vert z\right) .
\end{equation*}%
Apply (\ref{1F1_Laguerre})\ and simplify the result to complete the proof.
\end{proof}

\begin{remark}
Note that (\ref{2F2_theorem})\ provides an alternative expression to the one
found in the literature \cite{2F2JL}:%
\begin{equation*}
_{2}F_{2}\left( \left.
\begin{array}{c}
a+n,b+m \\
a,b%
\end{array}%
\right\vert z\right) =e^{z}\sum_{k=0}^{n+m}\frac{\left( -n-m\right)
_{k}\,\left( -z\right) ^{k}}{k!\,\left( a\right) _{k}}\,_{3}F_{2}\left(
\left.
\begin{array}{c}
-k,-m,b-a-n \\
b,-n-m%
\end{array}%
\right\vert 1\right) .
\end{equation*}
\end{remark}

\begin{theorem}
\label{Theorem: 3F3}For $n,m,k=0,1,2,\ldots $the following reduction formula
holds true:%
\begin{eqnarray*}
&&_{3}F_{3}\left( \left.
\begin{array}{c}
a+n,b+m,c+k \\
a,b,c%
\end{array}%
\right\vert z\right) \\
&=&\frac{m!\,e^{z}}{\left( b\right) _{m}\left( a\right) _{n}}\sum_{\ell
=0}^{k}\binom{k}{\ell }\frac{\Gamma \left( a+n+\ell \right) }{\left(
c\right) _{\ell }}\sum_{s=0}^{n}\binom{n}{s}\,z^{\ell +s}\,\frac{%
L_{m}^{b+\ell +s-1}\left( -z\right) }{\Gamma \left( a+\ell +s\right) }.
\end{eqnarray*}
\end{theorem}

\begin{proof}
According to (\ref{Main_theorem}), we have%
\begin{eqnarray*}
&&_{3}F_{3}\left( \left.
\begin{array}{c}
a+n,b+m,c+k \\
a,b,c%
\end{array}%
\right\vert z\right) \\
&=&\sum_{\ell =0}^{k}\binom{k}{\ell }\frac{\left( a+n\right) _{\ell
}\,\left( b+m\right) _{\ell }\,z^{\ell }}{\left( a\right) _{\ell }\left(
b\right) _{\ell }\left( c\right) _{\ell }}\,_{2}F_{2}\left( \left.
\begin{array}{c}
a+n+\ell ,b+m+\ell \\
a+\ell ,b+\ell%
\end{array}%
\right\vert z\right) .
\end{eqnarray*}%
Apply (\ref{2F2_theorem})\ and simplify the result to complete the proof.
\end{proof}

\begin{remark}
In \cite{2F2JL}, it is conjectured the following reduction formula:%
\begin{equation}
_{p}F_{p}\left( \left.
\begin{array}{c}
a_{1}+n_{1},\ldots ,a_{p}+n_{p} \\
a_{1},\ldots ,a_{p}%
\end{array}%
\right\vert z\right) =e^{z}\,\mathcal{P}_{p}\left( n_{1},\ldots
,n_{p};a_{1},\ldots ,a_{p}\right) ,  \label{pFp_general}
\end{equation}%
where $\mathcal{P}_{p}\left( n_{1},\ldots ,n_{p};a_{1},\ldots ,a_{p}\right) $
is a polynomial in $z$ of degree $n=n_{1}+\ldots +n_{p}$, being $%
n_{j}=0,1,2,\ldots $ $\left( j=1,\ldots ,p\right) $. From Theorems \ref%
{Theorem: 2F2} and \ref{Theorem: 3F3}\ is clear that the repetead
application of (\ref{Corollary: main})\ yields the form given in (\ref%
{pFp_general}).
\end{remark}

\section{Application to $_{p+1}F_{p}$ reduction formulas with arbitrary $p$
and $z$}

\label{Section: arbitrary p and z}

Recently, in \cite{pFqJL}\ we found the following reduction formula for $%
p=1,2,\ldots $ and $a_{i}\neq a_{j}$ $\left( i\neq j\right) $:%
\begin{eqnarray}
&&_{p+2}F_{p+1}\left( \left.
\begin{array}{c}
a_{1},\ldots ,a_{p},b,c+n \\
a_{1}+1,\ldots ,a_{p}+1,c%
\end{array}%
\right\vert z\right)  \label{Reduction_pFq_JL} \\
&=&\frac{z^{1-c}}{\left( c\right) _{n}}\prod_{s=1}^{p}a_{s}\sum_{k=0}^{n}%
\frac{\mathcal{H}^{n}\left( a_{k},1-b,n+c-a_{k}-1;z\right) }{\prod_{\ell
\neq k}^{p}\left( a_{\ell }-a_{k}\right) },  \notag
\end{eqnarray}%
where%
\begin{eqnarray}
&&\mathcal{H}^{n}\left( \alpha ,\beta ,\gamma ;z\right) :=\frac{d^{n}}{dz^{n}%
}\left[ z^{\gamma }\,\mathrm{B}_{z}\left( \alpha ,\beta \right) \right]
\label{H^n_def} \\
&=&\left( -1\right) ^{n}z^{\gamma -n}  \notag \\
&&\left\{ \left( -\gamma \right) _{n}\,\mathrm{B}_{z}\left( \alpha ,\beta
\right) -z^{\alpha }\sum_{k=0}^{n-1}\binom{n}{k+1}\left( -\gamma \right)
_{n-k-1}\left( 1-\alpha \right) _{k}\,_{2}F_{1}\left( \left.
\begin{array}{c}
\alpha ,1-\beta \\
\alpha -k%
\end{array}%
\right\vert z\right) \right\} ,  \notag
\end{eqnarray}%
which satisfies the following property for $\mathrm{Re}\beta >n$
\begin{equation}
\mathcal{H}^{n}\left( \alpha ,\beta ,\gamma ;1\right) =\left( -1\right)
^{n}\,\left( -\gamma \right) _{n}\,\mathrm{B}\left( \alpha ,\beta \right) .
\label{H^n(a,b,c,1)}
\end{equation}
Next, with the aid of Corollary \ref{Corollary: main}, we derive a much
simpler reduction formula for the same case as the one given in (\ref%
{Reduction_pFq_JL}).

\begin{theorem}
\label{Theorem: pFq_new_formulation}For $n=0,1,2,\ldots $, $p=1,2,\ldots $
and $a_{i}\neq a_{j}$ $\left( i\neq j\right) $, the following reduction
formula holds true:\
\begin{eqnarray}
&&_{p+2}F_{p+1}\left( \left.
\begin{array}{c}
a_{1},\ldots ,a_{p},b,c+n \\
a_{1}+1,\ldots ,a_{p}+1,c%
\end{array}%
\right\vert z\right)  \label{Theorem_p+2_F_p+1} \\
&=&\prod_{s=1}^{p}a_{s}\sum_{k=0}^{n}\binom{n}{k}\frac{\left( b\right) _{k}}{%
\left( c\right) _{k}}\sum_{\ell =1}^{p}\frac{z^{-a_{\ell }}\,\mathrm{B}%
_{z}\left( a_{\ell }+k,1-b-k\right) }{\prod_{j\neq \ell }^{p}\left(
a_{j}-a_{\ell }\right) }.  \notag
\end{eqnarray}
\end{theorem}

\begin{proof}
According to (\ref{Corollary_main}), we have%
\begin{eqnarray*}
&&_{p+2}F_{p+1}\left( \left.
\begin{array}{c}
a_{1},\ldots ,a_{p},b,c+n \\
a_{1}+1,\ldots ,a_{p}+1,c%
\end{array}%
\right\vert z\right) \\
&=&\sum_{k=0}^{n}\frac{\,\left( a_{1}\right) _{k}\cdots \left( a_{p}\right)
_{k}\,\left( b\right) _{k}\,z^{k}}{\left( a_{1}+1\right) _{k}\cdots \left(
a_{p}+1\right) _{k}\,\left( c\right) _{k}}\,_{p+1}F_{p}\left( \left.
\begin{array}{c}
a_{1}+k,\ldots ,\,a_{p}+k,\,b+k \\
a_{1}+1+k,\ldots ,\,a_{p}+1+k%
\end{array}%
\right\vert z\right) .
\end{eqnarray*}%
Take into account (\ref{(a)_k/(a+1)_k})\ and the results given in \cite[%
Eqns. 7.10.1(1)\&7.3.1(28)]{Prudnikov3}, i.e.
\begin{equation*}
_{m+1}F_{m}\left( \left.
\begin{array}{c}
a,b_{1},\ldots ,b_{m} \\
b_{1}+1,\ldots ,b_{m}+1%
\end{array}%
\right\vert z\right) =\prod_{s=1}^{m}b_{s}\sum_{j=1}^{m}\frac{z^{-b_{j}}\,%
\mathrm{B}_{z}\left( b_{j},1-a\right) }{\prod_{\ell \neq j}^{m}\left(
b_{\ell }-b_{j}\right) },
\end{equation*}%
to complete the proof.
\end{proof}

\begin{remark}
It is worth noting that the particular case $p=0$ is not included in (\ref%
{Theorem_p+2_F_p+1}), but it is given in the literature\ as \cite[Eqn.
7.3.1(21,140)]{Prudnikov3}:%
\begin{equation}
_{2}F_{1}\left( \left.
\begin{array}{c}
b,c+n \\
c%
\end{array}%
\right\vert z\right) =\left( 1-z\right) ^{-b}\sum_{k=0}^{n}\binom{n}{k}\frac{%
\left( b\right) _{k}}{\left( c\right) _{k}}\left( \frac{z}{1-z}\right) ^{k}.
\label{Reduction_2F1}
\end{equation}
\end{remark}

Despite the fact that (\ref{Theorem_p+2_F_p+1})\ is quite different from
expression reported in the literature, i.e. (\ref{Reduction_pFq_JL}), we can
derive the same formula reported in \cite{pFqJL}\ for $z=1$ from (\ref%
{Theorem_p+2_F_p+1}). Indeed, substitute $z=1$ in (\ref{Theorem_p+2_F_p+1}),
exchange the sum order, and expand the corresponding beta function according
to (\ref{beta_def}), to obtain
\begin{eqnarray*}
&&_{p+2}F_{p+1}\left( \left.
\begin{array}{c}
a_{1},\ldots ,a_{p},b,c+n \\
a_{1}+1,\ldots ,a_{p}+1,c%
\end{array}%
\right\vert 1\right) \\
&=&\prod_{s=1}^{p}a_{s}\sum_{k=0}^{n}\binom{n}{k}\frac{\left( b\right) _{k}}{%
\left( c\right) _{k}}\sum_{\ell =1}^{p}\frac{\,\mathrm{B}\left( a_{\ell
}+k,1-b-k\right) }{\prod_{j\neq \ell }^{p}\left( a_{j}-a_{\ell }\right) } \\
&=&\prod_{s=1}^{p}a_{s}\sum_{\ell =1}^{p}\frac{1}{\prod_{j\neq \ell
}^{p}\left( a_{j}-a_{\ell }\right) \,\Gamma \left( a_{\ell }+1-b\right) }%
\sum_{k=0}^{n}\binom{n}{k}\frac{\left( b\right) _{k}}{\left( c\right) _{k}}%
\,\Gamma \left( a_{\ell }+k\right) \,\Gamma \left( 1-b-k\right) .
\end{eqnarray*}%
Now, take into account the reflection formula of the gamma function (\ref%
{gamma_reflection}), the definition of the beta function (\ref{beta_def})
and the Pochhammer symbol (\ref{Pochhammer_def}), thus
\begin{eqnarray*}
&&_{p+2}F_{p+1}\left( \left.
\begin{array}{c}
a_{1},\ldots ,a_{p},b,c+n \\
a_{1}+1,\ldots ,a_{p}+1,c%
\end{array}%
\right\vert 1\right) \\
&=&\frac{\pi }{\Gamma \left( 1-b\right) \sin \pi b}\prod_{s=1}^{p}a_{s}%
\sum_{\ell =1}^{p}\frac{\Gamma \left( a_{\ell }\right) }{\prod_{j\neq \ell
}^{p}\left( a_{j}-a_{\ell }\right) \,\Gamma \left( a_{\ell }+1-b\right) }%
\sum_{k=0}^{n}\binom{n}{k}\frac{\left( a_{\ell }\right) _{k}}{\left(
c\right) _{k}}\,\left( -1\right) ^{k} \\
&=&\prod_{s=1}^{p}a_{s}\sum_{\ell =1}^{p}\frac{\mathrm{B}\left( a_{\ell
},1-b\right) }{\prod_{j\neq \ell }^{p}\left( a_{j}-a_{\ell }\right) \,}%
\sum_{k=0}^{n}\binom{n}{k}\frac{\left( a_{\ell }\right) _{k}}{\left(
c\right) _{k}}\,\left( -1\right) ^{k}.
\end{eqnarray*}%
Finally, apply the result given in (\ref{Chu_4F3}) to arrive at the desired
result given in \cite{pFqJL}, i.e.
\begin{eqnarray}
&&_{p+2}F_{p+1}\left( \left.
\begin{array}{c}
a_{1},\ldots ,a_{p},b,c+n \\
a_{1}+1,\ldots ,a_{p}+1,c%
\end{array}%
\right\vert 1\right)  \label{p+2_F_p+1(1)} \\
&=&\frac{1}{\left( c\right) _{n}}\prod_{s=1}^{p}a_{s}\sum_{\ell =1}^{p}\frac{%
\left( c-a_{\ell }\right) _{n}\,\mathrm{B}\left( a_{\ell },1-b\right) }{%
\prod_{j\neq \ell }^{p}\left( a_{j}-a_{\ell }\right) \,}.  \notag
\end{eqnarray}

\begin{remark}
As a consistency test, note that for $p=1$, (\ref{p+2_F_p+1(1)})\ reduces to%
\begin{equation*}
_{3}F_{2}\left( \left.
\begin{array}{c}
a,b,c+m \\
b+1,c%
\end{array}%
\right\vert 1\right) =b\,\mathrm{B}\left( 1-a,b\right) \,\frac{\left(
c-b\right) _{m}}{\left( c\right) _{m}},
\end{equation*}%
which coincides with (\ref{4F3(1)_theorem})\ for $n=1$.
\end{remark}

\begin{theorem}
For $n,m=0,1,2,\ldots $, $p=1,2,\ldots $ and $a_{i}\neq a_{j}$ $\left( i\neq
j\right) $, the following reduction formula holds true:%
\begin{eqnarray}
&&_{p+3}F_{p+2}\left( \left.
\begin{array}{c}
a_{1},\ldots ,a_{p},b,c+n,d+m \\
a_{1}+1,\ldots ,a_{p}+1,c,d%
\end{array}%
\right\vert z\right)  \label{Theorem_p+3_F_p+2} \\
&=&\frac{z^{1-d}\prod_{s=1}^{p}a_{s}}{\left( d\right) _{m}}\sum_{k=0}^{n}%
\binom{n}{k}\frac{\left( b\right) _{k}}{\left( c\right) _{k}}\sum_{\ell
=1}^{p}\frac{\mathcal{H}^{m}\left( a_{\ell }+k,1-b-k,m+d-a_{\ell
}-1;z\right) }{\prod_{j\neq \ell }^{p}\left( a_{j}-a_{\ell }\right) }.
\notag
\end{eqnarray}
\end{theorem}

\begin{proof}
Apply the $n$-th derivative to both sides of (\ref{Theorem_p+2_F_p+1}),
taking into account (\ref{Diff_pFq}) and (\ref{H^n_def}) to complete the
proof.
\end{proof}

\bigskip

Alternatively, we can derive a much simpler expression for the reduction
formula given in (\ref{Theorem_p+3_F_p+2}). \bigskip

\begin{theorem}
For $n,m=0,1,2,\ldots $, $p=1,2,\ldots $ and $a_{i}\neq a_{j}$ $\left( i\neq
j\right) $, the following reduction formula holds true:%
\begin{eqnarray}
&&_{p+3}F_{p+2}\left( \left.
\begin{array}{c}
a_{1},\ldots ,a_{p},b,c+n,d+m \\
a_{1}+1,\ldots ,a_{p}+1,c,d%
\end{array}%
\right\vert z\right)  \label{Theorem_p+3_F_p+2_new} \\
&=&\frac{\prod_{s=1}^{p}a_{s}}{\left( d\right) _{m}}\sum_{k=0}^{n}\frac{%
\binom{n}{k}}{\left( c\right) _{k}}\sum_{s=0}^{m}\binom{m}{s}\left(
b+k\right) _{s}\left( d+k+s\right) _{m-s}  \notag \\
&&\sum_{\ell =1}^{p}\frac{\mathrm{B}_{z}\left( a_{\ell }+k+s,1-b-k-s\right)
}{z^{a_{\ell }}\prod_{j\neq \ell }^{p}\left( a_{j}-a_{\ell }\right) }.
\notag
\end{eqnarray}
\end{theorem}

\begin{proof}
First, note that for $p=1$, (\ref{Theorem_p+2_F_p+1})\ reduces to%
\begin{equation}
_{3}F_{2}\left( \left.
\begin{array}{c}
a,b,c+n \\
a+1,c%
\end{array}%
\right\vert z\right) =a\,z^{-a}\sum_{k=0}^{n}\binom{n}{k}\frac{\left(
b\right) _{k}}{\left( c\right) _{k}}\,\mathrm{B}_{z}\left( a_{\ell
}+k,1-b-k\right) .  \label{3F2_resultado}
\end{equation}%
Next, apply the differentiation formula (\ref{Diff_pFq})\ to (\ref%
{Theorem_p+2_F_p+1}) in order to obtain%
\begin{eqnarray}
&&_{p+3}F_{p+2}\left( \left.
\begin{array}{c}
a_{1},\ldots ,a_{p},b,c+n,d+m \\
a_{1}+1,\ldots ,a_{p}+1,c,d%
\end{array}%
\right\vert z\right)  \notag \\
&=&\frac{z^{1-d}\prod_{s=1}^{p}a_{s}}{\left( d\right) _{m}}\sum_{k=0}^{n}%
\binom{n}{k}\frac{\left( b\right) _{k}}{\left( c\right) _{k}}\sum_{\ell
=1}^{p}\frac{1}{\prod_{j\neq \ell }^{p}\left( a_{j}-a_{\ell }\right) }
\notag \\
&&\frac{d^{m}}{dz^{m}}\left[ z^{m+d-1-a_{\ell }}\,\,\mathrm{B}_{z}\left(
a_{\ell }+k,1-b-k\right) \right]  \label{Dm[z*beta]}
\end{eqnarray}%
Now, according to \cite[Eqn. 7.3.1(28)]{Prudnikov3}, we have the reduction
formula:
\begin{equation*}
_{2}F_{1}\left( \left.
\begin{array}{c}
a,b \\
a+1%
\end{array}%
\right\vert z\right) =a\,z^{-a}\,\mathrm{B}_{z}\left( a,1-b\right) \,,
\end{equation*}%
thus applying the differentiation formula (\ref{Diff_pFq_general}) and the
result given in (\ref{3F2_resultado}), we arrive at%
\begin{eqnarray}
&&\frac{d^{m}}{dz^{m}}\left[ z^{m+d-1-a_{\ell }}\,\,\mathrm{B}_{z}\left(
a_{\ell }+k,1-b-k\right) \right]  \notag \\
&=&\frac{1}{a_{\ell }+k}\frac{d^{m}}{dz^{m}}\left[ z^{m+d-1+k}\,\,_{2}F_{1}%
\left( \left.
\begin{array}{c}
a,b \\
a+1%
\end{array}%
\right\vert z\right) \right]  \notag \\
&=&\frac{\left( d+k\right) _{m}}{a_{\ell }+k}\,\,z^{d-1+k}\,_{3}F_{2}\left(
\left.
\begin{array}{c}
m+d-1+k,a_{\ell }+k,b+k \\
d+k,a_{\ell }+k+1%
\end{array}%
\right\vert z\right)  \notag \\
&=&\left( d+k\right) _{m}\,\,z^{d-1}\sum_{s=0}^{m}\binom{m}{s}\frac{\left(
b+k\right) _{s}}{\left( d+k\right) _{s}}\,z^{-a_{\ell }}\,\mathrm{B}%
_{z}\left( a_{\ell }+k+s,1-b-k-s\right) .  \label{Dm[z*beta]_resultado}
\end{eqnarray}%
Insert (\ref{Dm[z*beta]_resultado})\ in (\ref{Dm[z*beta]})\ and simplify the
result, taking into account the property%
\begin{equation*}
\frac{\Gamma \left( \alpha +m\right) }{\Gamma \left( \alpha +s\right) }%
=\left( \alpha \right) _{m-s},
\end{equation*}%
to complete the proof.
\end{proof}

\begin{theorem}
For $p=1,2,\ldots $, $a_{i}\neq a_{j}$ $\left( i\neq j\right) $, and $%
\mathrm{Re}b<1-\max \left( n,m\right) $, the following reduction formula
holds true:%
\begin{eqnarray}
&&_{p+3}F_{p+2}\left( \left.
\begin{array}{c}
a_{1},\ldots ,a_{p},b,c+n,d+m \\
a_{1}+1,\ldots ,a_{p}+1,c,d%
\end{array}%
\right\vert 1\right)  \label{Theorem_p+3_F_p+2_(1)} \\
&=&\frac{\prod_{s=1}^{p}a_{s}}{\left( d\right) _{m}\left( c\right) _{n}}%
\sum_{\ell =1}^{p}\frac{\left( d-a_{\ell }\right) _{m}\left( c-a_{\ell
}\right) _{n}}{\prod_{j\neq \ell }^{p}\left( a_{j}-a_{\ell }\right) }\,%
\mathrm{B}\left( a_{\ell },1-b\right) .  \notag
\end{eqnarray}
\end{theorem}

\begin{proof}
Taking $z=1$ in (\ref{Theorem_p+3_F_p+2}), and applying (\ref{H^n(a,b,c,1)}%
), we obtain
\begin{eqnarray*}
&&_{p+3}F_{p+2}\left( \left.
\begin{array}{c}
a_{1},\ldots ,a_{p},b,c+n,d+m \\
a_{1}+1,\ldots ,a_{p}+1,c,d%
\end{array}%
\right\vert 1\right) \\
&=&\frac{1}{\left( d\right) _{m}}\prod_{s=1}^{p}a_{s}\sum_{k=0}^{n}\binom{n}{%
k}\frac{\left( b\right) _{k}}{\left( c\right) _{k}}\sum_{\ell =1}^{p}\frac{%
\left( d-a_{\ell }\right) _{m}\,\mathrm{B}\left( a_{\ell }+k,1-b-k\right) }{%
\prod_{j\neq \ell }^{p}\left( a_{j}-a_{\ell }\right) }.
\end{eqnarray*}%
Now, exchange the sum order and expand the beta function according to (\ref%
{beta_def}),
\begin{eqnarray*}
&&_{p+3}F_{p+2}\left( \left.
\begin{array}{c}
a_{1},\ldots ,a_{p},b,c+n,d+m \\
a_{1}+1,\ldots ,a_{p}+1,c,d%
\end{array}%
\right\vert 1\right) \\
&=&\frac{\prod_{s=1}^{p}a_{s}}{\left( d\right) _{m}}\sum_{\ell =1}^{p}\frac{%
\left( d-a_{\ell }\right) _{m}}{\prod_{j\neq \ell }^{p}\left( a_{j}-a_{\ell
}\right) \,\Gamma \left( a_{\ell }+1-b\right) }\sum_{k=0}^{n}\binom{n}{k}%
\frac{\left( b\right) _{k}}{\left( c\right) _{k}}\,\Gamma \left( a_{\ell
}+k\right) \,\Gamma \left( 1-b-k\right) .
\end{eqnarray*}%
Apply the reflection formula (\ref{gamma_reflection})\ and the definition of
the Pochhammer symbol (\ref{Pochhammer_def})\ to arrive at%
\begin{eqnarray}
&&_{p+3}F_{p+2}\left( \left.
\begin{array}{c}
a_{1},\ldots ,a_{p},b,c+n,d+m \\
a_{1}+1,\ldots ,a_{p}+1,c,d%
\end{array}%
\right\vert 1\right)  \label{F_intermedios} \\
&=&\frac{\pi \prod_{s=1}^{p}a_{s}}{\sin \pi b\,\Gamma \left( b\right)
\,\left( d\right) _{m}}\sum_{\ell =1}^{p}\frac{\left( d-a_{\ell }\right)
_{m}\,\Gamma \left( a_{\ell }\right) }{\prod_{j\neq \ell }^{p}\left(
a_{j}-a_{\ell }\right) \,\Gamma \left( a_{\ell }+1-b\right) }\sum_{k=0}^{n}%
\binom{n}{k}\frac{\left( a_{\ell }\right) _{k}}{\left( c\right) _{k}}%
\,\left( -1\right) ^{k}.  \notag
\end{eqnarray}%
Now, insert (\ref{Chu_4F3}) in (\ref{F_intermedios})\ and take into account
the reflection formula (\ref{gamma_reflection})\ and the definition of the
beta function (\ref{beta_def}) to complete the proof.
\end{proof}

\bigskip

Note that the particular case $p=0$ is not included in (\ref%
{Theorem_p+3_F_p+2}). In order to derive the corresponding formula for $p=0$%
, we first prove the following lemma.

\bigskip

\begin{lemma}
For $m=0,1,2,\ldots $ and $\alpha ,\beta \in
\mathbb{R}
$, the following derivative formula holds true:%
\begin{eqnarray}
&&\frac{d^{m}}{dz^{m}}\left( \frac{z^{\alpha }}{\left( 1-z\right) ^{\beta }}%
\right)  \label{Lema} \\
&=&\frac{z^{\alpha -m}}{\left( 1-z\right) ^{\beta +m}}\left( \alpha
-m+1\right) _{m}\,_{2}F_{1}\left( \left.
\begin{array}{c}
-m,1-m+\alpha -\beta \\
1-m+\alpha%
\end{array}%
\right\vert z\right) .  \notag
\end{eqnarray}
\end{lemma}

\begin{proof}
Apply Leibinz's differentiation formula \cite[Eqn. 1.4.2]{NIST}%
\begin{equation*}
\frac{d^{n}}{dz^{n}}\left[ f\left( z\right) \,g\left( z\right) \right]
=\sum_{k=0}^{n}\binom{n}{k}\frac{d^{k}}{dz^{k}}\left[ f\left( z\right) %
\right] \,\frac{d^{n-k}}{dz^{n-k}}\left[ g\left( z\right) \right] ,
\end{equation*}%
and the $n$-th derivative formula \cite[Eqn. 1.1.2(1)]{Brychov}%
\begin{equation*}
\frac{d^{n}}{dz^{n}}\left[ z^{\alpha }\right] =\left( -1\right) ^{n}\left(
-\alpha \right) ^{n}z^{\alpha -n},
\end{equation*}%
hence%
\begin{equation*}
\frac{d^{n}}{dz^{n}}\left[ \left( 1-z\right) ^{\alpha }\right] =\left(
-\alpha \right) ^{n}\left( 1-z\right) ^{\alpha -n},
\end{equation*}%
to obtain%
\begin{eqnarray}
&&\frac{d^{m}}{dz^{m}}\left( \frac{z^{\alpha }}{\left( 1-z\right) ^{\beta }}%
\right)  \notag \\
&=&\sum_{\ell =0}^{m}\binom{m}{\ell }\frac{d^{\ell }}{dz^{\ell }}\left[
\left( 1-z\right) ^{-\beta }\right] \,\frac{d^{m-\ell }}{dz^{m-\ell }}\left[
z^{\alpha }\right]  \notag \\
&=&\frac{\left( -1\right) ^{m}z^{\alpha -m}}{\left( 1-z\right) ^{\beta }}%
\sum_{\ell =0}^{m}\binom{m}{\ell }\left( \beta \right) _{\ell }\left(
-\alpha \right) _{m-\ell }\left( \frac{z}{z-1}\right) ^{\ell }.  \label{D^m}
\end{eqnarray}%
Now, according to (\ref{Hypergeometric_sum})\ and (\ref{(-m)_k}), we have%
\begin{eqnarray}
_{2}F_{1}\left( \left.
\begin{array}{c}
-m,\beta \\
1-m+\alpha%
\end{array}%
\right\vert z\right) &=&\sum_{\ell =0}^{\infty }\frac{\left( -m\right)
_{\ell }\left( \beta \right) _{\ell }}{\ell !\,\left( 1-m+\alpha \right)
_{\ell }}z^{\ell }  \notag \\
&=&\sum_{\ell =0}^{m}\binom{m}{\ell }\frac{\left( \beta \right) _{\ell
}\,\left( -z\right) ^{\ell }}{\left( 1-m+\alpha \right) _{\ell }}.
\label{2F1(-m)}
\end{eqnarray}%
However, from the definition of the Pochhammer symbol (\ref{Pochhammer_def})
and the property (\ref{(-x)_n}),\ it is easy to prove\ the identity%
\begin{equation}
\left( -\alpha \right) _{m-\ell }=\frac{\left( -1\right) ^{m+\ell }\left(
\alpha -m+1\right) _{m}}{\left( 1-m+\alpha \right) _{\ell }},
\label{Identity}
\end{equation}%
thus, substituting (\ref{Identity})\ in (\ref{D^m})\ and taking into account
(\ref{2F1(-m)}), we arrive at%
\begin{equation*}
\frac{d^{m}}{dz^{m}}\left( \frac{z^{\alpha }}{\left( 1-z\right) ^{\beta }}%
\right) =\frac{z^{\alpha -m}\left( \alpha -m+1\right) _{m}}{\left(
1-z\right) ^{\beta }}\,_{2}F_{1}\left( \left.
\begin{array}{c}
-m,\beta \\
1-m+\alpha%
\end{array}%
\right\vert \frac{z}{z-1}\right) .
\end{equation*}%
Finally, apply the linear transformation formula (\ref{2F1_linear})\ to
complete the proof.
\end{proof}

\begin{theorem}
For $b,c,d\in
\mathbb{C}
$, $n,m=0,1,2,\ldots $, and $z\neq 1$, the following reduction formula holds
true:%
\begin{eqnarray}
&&_{3}F_{2}\left( \left.
\begin{array}{c}
b,c+n,d+m \\
c,d%
\end{array}%
\right\vert z\right)  \label{3F2_p=0} \\
&=&\frac{1}{\left( 1-z\right) ^{b+m}}\sum_{k=0}^{n}\binom{n}{k}\frac{\left(
b\right) _{k}\,\left( d+m\right) _{k}}{\left( c\right) _{k}\,\left( d\right)
_{k}}\left( \frac{z}{1-z}\right) ^{k}\,_{2}F_{1}\left( \left.
\begin{array}{c}
-m,d-b \\
d+k%
\end{array}%
\right\vert z\right) .  \notag
\end{eqnarray}
\end{theorem}

\begin{proof}
According to the differentiation formula (\ref{Diff_pFq})\ and the reduction
formula (\ref{Reduction_2F1}), we have
\begin{equation*}
_{3}F_{2}\left( \left.
\begin{array}{c}
b,c+n,d+m \\
c,d%
\end{array}%
\right\vert z\right) =\frac{z^{1-d}}{\left( d\right) _{m}}\sum_{k=0}^{n}%
\binom{n}{k}\frac{\left( b\right) _{k}}{\left( c\right) _{k}}\frac{d^{m}}{%
dz^{m}}\left[ \frac{z^{k+m+d-1}}{\left( 1-z\right) ^{b+k}}\right] .
\end{equation*}%
Apply the differentiation formula (\ref{Lema}) and take into account the
property%
\begin{equation*}
\frac{\left( d+k\right) _{m}}{\left( d\right) _{m}}=\frac{\left( d+m\right)
_{k}}{\left( d\right) _{k}},
\end{equation*}%
to complete the proof.
\end{proof}

\begin{remark}
Note that
\begin{equation*}
\lim_{z\rightarrow 1}\,\left\vert _{3}F_{2}\left( \left.
\begin{array}{c}
b,c+n,d+m \\
c,d%
\end{array}%
\right\vert z\right) \right\vert =\infty .
\end{equation*}
\end{remark}

\begin{remark}
As a consistency test, note that (\ref{3F2_p=0})\ is reduced to (\ref%
{Reduction_2F1}), taking $b=d$.
\end{remark}

\section{Conclusions}

\label{Section: conclusions}

Throughout the paper, we have systematically used the hypergeometric
identity (\ref{Corollary: main}), i.e.
\begin{eqnarray}
&&_{p+1}F_{q+1}\left( \left.
\begin{array}{c}
a_{1},\ldots ,a_{p},c+n \\
b_{1},\ldots ,b_{q},c%
\end{array}%
\right\vert z\right)  \label{Corollary_main_conclusions} \\
&=&\sum_{k=0}^{n}\binom{n}{k}\frac{\,\left( a_{1}\right) _{k}\cdots \left(
a_{p}\right) _{k}\,\,z^{k}}{\,\left( b_{1}\right) _{k}\cdots \left(
b_{q}\right) _{k}\,}\,_{p}F_{q}\left( \left.
\begin{array}{c}
a_{1}+k,\ldots ,\,a_{p}+k \\
b_{1}+k,\ldots ,\,b_{q}+k%
\end{array}%
\right\vert z\right) ,  \notag
\end{eqnarray}%
to obtain a set of summation formulas that do not seem to be reported in the
existing literature. In order to see how this method works, consider a
vector of parameters $\vec{\alpha}=\left( \alpha _{1},\ldots ,\alpha _{\ell
}\right) $ and adopt the notation, $\vec{\alpha}+k=\left( \alpha
_{1}+k,\ldots ,\alpha _{\ell }+k\right) $. If we know the reduction formula
of a generalized hypergeometric function of the form:%
\begin{equation*}
_{p}F_{q}\left( \left.
\begin{array}{c}
a_{1}\left( \vec{\alpha}\right) ,\ldots ,\,a_{p}\left( \vec{\alpha}\right)
\\
b_{1}\left( \vec{\alpha}\right) ,\ldots ,\,b_{q}\left( \vec{\alpha}\right)%
\end{array}%
\right\vert z\right) ,
\end{equation*}%
in such a way that $\forall i=1,\ldots ,p$ and $\forall j=1,\ldots ,q$ we
have
\begin{equation*}
\left\{
\begin{array}{c}
a_{i}\left( \vec{\alpha}+k\right) =a_{i}\left( \vec{\alpha}\right) +k, \\
b_{j}\left( \vec{\alpha}+k\right) =b_{j}\left( \vec{\alpha}\right) +k,%
\end{array}%
\right.
\end{equation*}%
then, according to (\ref{Corollary_main_conclusions}), we obtain $\forall
n=0,1,2,\ldots $ a reduction formula for the generalized hypergeometric
function:
\begin{equation*}
_{p+1}F_{q+1}\left( \left.
\begin{array}{c}
a_{1}\left( \vec{\alpha}\right) ,\ldots ,\,a_{p}\left( \vec{\alpha}\right)
,c+n \\
b_{1}\left( \vec{\alpha}\right) ,\ldots ,\,b_{q}\left( \vec{\alpha}\right) ,c%
\end{array}%
\right\vert z\right) .
\end{equation*}

It is worth noting that during the development of the proofs, we have found
some interesting formulas such as the recursive formula (\ref{Beta_recursive}%
) or the proof of the conjecture given in (\ref{pFp_general}). In addition,
we have proved in (\ref{Sum_Psi}) and (\ref{Sum_Psi_simple}) two finite sums
involving the psi function. Also, with the new reformulation given in (\ref%
{Theorem_p+2_F_p+1})\ of the reduction formula given in the literature, i.e.
(\ref{Reduction_pFq_JL}), we were able to obtain two new equivalent
reduction formulas in (\ref{Theorem_p+3_F_p+2})\ and (\ref%
{Theorem_p+3_F_p+2_new}). Finally, for the special case given in (\ref%
{3F2_p=0}), which is not included in (\ref{Theorem_p+3_F_p+2}), we have
developed a particular proof.

\end{document}